\documentclass[10pt,a4paper, reqno]{amsart}

\usepackage[utf8]{inputenc}
\usepackage{array,amsbsy,amscd,amsfonts,color,amssymb,amstext,amsmath,latexsym,amsthm,amscd,amsxtra,multicol,graphicx, tikz, physics}
\usepackage[english]{babel}
\usepackage{hyperref}
\usepackage[mathscr]{eucal}
\usepackage{tikz-cd}
\usepackage{comment}
\usepackage{enumitem}
\usepackage{seqsplit}
\usepackage{setspace}

\hypersetup{colorlinks=true,linkcolor=red, anchorcolor=green, citecolor=cyan, urlcolor=red, filecolor=magenta, pdftoolbar=true}

\numberwithin{equation}{section}

\newtheorem{theorem}{Theorem}
\numberwithin{theorem}{section}
\newtheorem{proposition}[theorem]{Proposition}
\newtheorem{lemma}[theorem]{Lemma}
\newtheorem{corollary}[theorem]{Corollary}
\newtheorem{question}[theorem]{Question}

\theoremstyle{definition}\newtheorem{definition}[theorem]{Definition}
\theoremstyle{definition}\newtheorem{remark}[theorem]{Remark}
\theoremstyle{definition}\newtheorem{example}[theorem]{Example}
\theoremstyle{definition}\newtheorem*{notation*}{Notation}
\theoremstyle{definition}\newtheorem*{convention*}{Convention}
\theoremstyle{definition}
\theoremstyle{definition}\newtheorem*{acknowledgment*}{Acknowledgments}

\newcommand{\N}{\mathbb{N}}
\newcommand{\R}{\mathbb{R}}
\newcommand{\C}{\mathbb{C}}
\newcommand{\A}{\mathcal{A}}
\newcommand{\B}{\mathcal{B}}
\newcommand{\h}{\mathcal{H}}
\newcommand{\K}{\mathcal{K}}
\newcommand{\E}{\mathcal{E}}
\newcommand{\F}{\mathcal{F}}
\newcommand{\M}{\mathcal{M}}
\newcommand{\T}{\mathbb{T}}
\newcommand{\bh}{\mathcal{B}(\mathcal{H})}
\newcommand{\bk}{\mathcal{B}(\mathcal{K})}

\DeclareMathOperator{\Span}{span}
\DeclareMathOperator{\ospan}{\overline{span}}
\DeclareMathOperator{\ran}{\mathscr{R}}

\DeclareMathOperator{\Lat}{Lat}

\DeclareMathOperator{\Alg}{Alg}

\begin{document}

\title{ Lattices of Logmodular algebras}

\author{B.V. Rajarama Bhat}
\email{bhat@isibang.ac.in}

\author{Manish Kumar}
\email{manish\_rs@isibang.ac.in}

\address{Indian Statistical Institute, Stat-Math. Unit, R V College Post, Bengaluru 560059, India}

\begin{abstract}
A subalgebra $\mathcal{A}$ of a $C^*$-algebra $\mathcal{M}$ is logmodular (resp. has factorization) if the set $\{a^*a; a\text{ is invertible with }a,a^{-1}\in\mathcal{A}\}$ is  dense in  (resp. equal to) the set of all positive and invertible elements of $\mathcal{M}$. There are large classes of well studied algebras, both in commutative and non-commutative settings, which are known to be logmodular. In this paper, we show that the lattice of projections in a von Neumann algebra $\mathcal{M}$ whose ranges are  invariant under a logmodular algebra in $\mathcal{M}$,  is a commutative subspace lattice. Further, if $\mathcal{M}$ is a factor then this lattice is a nest.  As a special case, it follows that all reflexive (in particular, completely distributive CSL) logmodular subalgebras of type I factors are nest algebras, thus answering a question  of  Paulsen and Raghupathi  \cite{PaRa}. 
We also discuss some sufficient criteria under which an algebra having factorization is automatically reflexive and is a nest algebra.
\end{abstract}

\keywords{Logmodular algebra, factorization, nest, commutative subspace lattice, nest algebra, CSL algebra}
\subjclass[2020]{Primary:   47L35. Secondary: 47L30, 46K50}

\maketitle

\section{Introduction}
The well-known Cholesky factorization theorem states that any  positive and invertible  $n\times n$ matrix can be written as $U^*U$ for some invertible upper triangular $n\times n$  matrix $U$ (so the inverse $U^{-1}$ is also upper triangular). We then say that the algebra of upper triangular matrices has {\em factorization} in $M_n$, the algebra of all $n\times n$ complex matrices. Using this or otherwise, one can show that an algebra consisting of block upper-triangular matrices (with respect to some orthonormal basis) also admits such factorization in $M_n$. Is there any other algebra in $M_n$ with this property? Paulsen and Raghupathi \cite{PaRa} showed that any algebra in $M_n$ containing all diagonal matrices, has factorization in $M_n$ if and only if it is unitarily equivalent to an algebra of block upper triangular matrices (they actually studied the notion called logmodular algebra, which is our interest in this paper as well). The general characterization was settled by Juschenko \cite{Ju} who showed that upto a change of basis, all algebras  with factorization in $M_n$ contain the diagonal algebra, thus showing that algebras of block upper triangular matrices are all which have factorization.

A natural question that arises is up to what extent these results generalize to infinite-dimensional setting. Let $\h$ be a complex and separable Hilbert space, and let $\bh$ denote the algebra of all bounded operators on $\h$. Let $\E$  be a collection of closed subspaces of $\h$ totally ordered with respect to inclusion (such  a collection is called a {\em nest}), and let $\Alg\E$ (called {\em nest algebra})  denote the algebra of all operators in $\bh$ which leave all subspaces in $\E$  invariant. Note that an algebra of block upper triangular matrices is nothing but $\Alg\E$  for some nest $\E$  in $\C^n$, and vice versa. So a generalization of Cholesky factorization  would be to ask whether $\Alg\E$  has factorization  in $\bh$  for a nest $\E$ in $\h$. Gohberg and Krein \cite{GoKr} appear to be the first who studied factorization  along nest algebras,  mainly examining  positive and invertible operators  `close' to identity operator. Arveson \cite{Ar2} considered factorization property of nest algebras arising out of nests of order type $\mathbb{Z}$. It was Larson \cite{La} who investigated factorization  property of arbitrary nest algebras, and he proved in particular that if $\E$ is a countable complete nest in $\h$, then $\Alg\E$ has factorization in $\bh$. Again what can be said about the converse? That is, if a subalgebra $\A$ has factorization in $\bh$, are they of the form $\Alg\E$ for some countable complete  nest  $\E$? In this paper, we show that this is indeed the case if we also assume that $\A$ is reflexive.

More generally, factorization property of  subalgebras of  arbitrary von Neumann algebras are considered. A classical result says that the Hardy algebra $\h^\infty(\T)$ on the unit circle has factorization  in $L^\infty(\T)$ i.e. for any non-negative element $f\in L^\infty(\T)$ with $1/f\in L^\infty(\T)$, there is an element $h\in\h^\infty(\T)$ with $1/h\in\h^\infty(\T)$ such that $f=\bar{h}h$. Some other function algebras like weak$^*$-Dirichlet algebras introduced by Srinivasan and Wang \cite{SrWa} have factorization property. Taking cue from analytic function algebras,  Arveson \cite{Ar1}  introduced the theory of finite maximal subdiagonal algebras as  noncommutative variant and considered many results analogous to the classical Hardy space theory, showing in particular that they have factorization property. Later several authors have examined such algebras  in different settings. For more about  algebras with factorization  see \cite{Ar1, Ar2, La, La1, Pi, So, McMuSa, Da, GoKr}, and for some closely related  properties see \cite{Po, Po1, AnKa, JiSa,  McMuSa, So1}  to name a few.

An algebra  with factorization property   is a particular case of  {\em logmodular algebras}. The notion of logmodularity was first introduced by Hoffman \cite{Ho} for subalgebras of commutative $C^*$-algebras, whose main idea was to generalize some classical results of analytic function theory in the unit disc. Blecher and Labuschagne \cite{BlLa} extended this notion  to subalgebras of non-commutative $C^*$-algebras. They studied completely contractive representations on such algebras and their extension properties. Paulsen and Raghupathi \cite{PaRa}  also studied representations of logmodular algebras and explored conditions under which a contractive representation is automatically completely contractive. In \cite{Ju}, Juschenko gave a complete characterization of all  logmodular  subalgebras of $M_n$. See \cite{BlLa1} for a  beautiful survey on logmodular algebras  arising out of tracial subalgebras and their relation to finite subdiagonal algebras among others. They show how most results generalized in 1960's from the Hardy space on the unit disc to more general function algebras generalize further to the non-commutative situation, though more sophisticated proof techniques had to be developed for the purpose.  We list some additional references on logmodular algebras in \cite{Ho, HoRo, FoSu, BlLa, BlLa1, PaRa, Ju}.

In this article, our  aim is to understand the behaviour of lattice of subspaces (or projections) invariant under logmodular algebras, and use it to  characterize  reflexive logmodular algebras. The main result answers a conjecture by  \cite{PaRa} in the affirmative, which asks whether every completely distributive CSL logmodular algebra of $\bh$ is a nest algebra. In fact, we show more generally that any $\M$-reflexive logmodular subalgebra of a factor $\M$ is a nest subalgebra. Moreover, we explore some sufficient criteria under which an algebra with factorization is automatically reflexive and  is a nest algebra. In particular it is proved that a subalgebra of $\bh$ with factorization, whose lattice consists of finite dimensional atoms, is reflexive and so it is a nest algebra. Also we show that any subalgebra with factorization  in a finite dimensional von Neumann algebra  must be a nest  subalgebra. Finally we give an example of a subalgebra in a  von Neumann algebra (certainly infinite dimensional), which has factorization but it is not a nest subalgebra.

\section{Definitions, examples and main results}

All Hilbert spaces considered in this paper are complex and separable. Throughout  $\bh$  denotes the algebra of all bounded operators on a Hilbert space  $\h$. By subspaces, projections and operators on $\h$, we mean closed subspaces, orthogonal projections and bounded operators  respectively. For any subspace  $F$ contained in a subspace $E$, we denote the orthogonal complement of $F $ in $E$ by $E\ominus F$. We write  $E^\perp$ for the orthogonal complement  of $E$ i.e. $E^\perp=\h\ominus E$. The projection onto a subspace $E$ is denoted by $P_E$. For any projection $p$, we write $p^\perp$ for the projection $1-p$, where $1$ is the identity operator on $\h$.   If $\{p_i\}_{i\in \Lambda}$ is a collection of projections, then $\vee_{i\in \Lambda}p_i$ denotes the projection onto the smallest subspace containing ranges of all $p_i's$,  and $\wedge_{i\in \Lambda}p_i$ denotes the projection onto the intersection of ranges of all $p_i's$. For any operator $x$ in $\bh$, $\ker x$ and $\ran(x)$ denote the kernel and range of $x$ respectively. All algebras considered will be subalgebras of  $\bh$, which are always assumed to be norm closed, and contain the identity operator which we shall denote by $1$.  Unless said otherwise, convergence of any sequence of operators is taken in norm topology.

We briefly recall some basic notions of  von Neumann algebra theory. A {\em von Neumann algebra}   is a self-adjoint subalgebra of $\bh$ containing $1$ and  closed in weak operator topology (WOT). A von Neumann algebra $\M$  is called  {\em factor} if its center $\M\cap\M'$ is trivial. Here $\M'$ denotes the commutant of $\M$ in $\bh$. Let $p,q$ be two projections in $\M$. Then $p$ and $q$  are said to be {\em (Murray-von Neumann) equivalent}, and denoted $p\sim q$, if there exists a partial isometry $v\in\M$  such that $v^*v=p$ and $vv^*=q$. We say $p\preceq q$ if there is a projection $q_1\in\M$  such that $q_1\leq q$ and $p\sim q_1$. Here $\leq$ denotes the usual order of self-adjoint operators, and $<$ will denote the strict order. A projection $p\in\M$  is called {\em finite} if the only projection $q$ in $\M$ such that $q\leq p$ and $q\sim p$ is $p$. The von Neumann algebra $\M$  is called {\em finite} if $1\in\M$  is finite. Note that if $p$ is a projection in $\M$, then $p\M p$ is a von Neumann algebra which is $*$-isomorphic to a von Neumann subalgebra of $\B(\K)$, where $\K$ is the range subspace of the projection $p$. See \cite{Co} for more details on these topics.

We now define some notations relevant to our results. Fix a  von Neumann algebra $\M$, which is always assumed to be  acting on a separable Hilbert space. Let $\A$ be a norm closed subalgebra (not necessarily self-adjoint) of $\M$. We denote by $\A^*$  the set $\{x\in \M; x^*\in\A\}$, and   by $\A^{-1}$ the set $\{x\in\A; x\mbox{ is invertible with }x^{-1}\in\A\}$. Let $\M_+^{-1}$ denote the set of all positive and invertible elements of $\M$. Note that all these notations make sense for any $C^*$-algebra. But our main focus in this paper is on von Neumann algebras.

For $\M$ and $\A$ as above, let $\Lat_\M\A$ denote the lattice of all projections in $\M$ whose ranges are invariant under every operator in $\A$ i.e.
\begin{align*}
    \Lat_\M\A=\{p\in\M; p=p^2=p^*\mbox{ and }  ap=pap \;\forall a\in\A\}.
\end{align*}
If $\M=\bh$, we denote $\Lat_\M\A$ simply by $\Lat\A$. Note that  if $\A$ is  also  considered as a subalgebra of $\bh$ (where $\M\subseteq\bh$), then we have $\Lat_\M\A=\M\cap\Lat\A$. Also note that $0,1\in\Lat_\M\A$ and $\Lat_\M\A$  is closed under the operations $\vee$ and $\wedge$ of arbitrary collection as well as  closed under strong operator topology (SOT).

Dually, let $\E$ be a collection of projections  in $\M$ (not necessarily a lattice), and let $\Alg_\M\E$ (or  $\Alg\E$ when $\M=\bh$) denote the algebra of all operators in $\M$  which leave range of every projection of $\E$ invariant i.e.
\begin{align*}
    \Alg_\M\E=\{x\in\M; xp=pxp\;\; \forall p\in\E\}.
\end{align*}
Again we note that $\Alg_\M\A=\M\cap \Alg\E$. Also it is clear that $\Alg_\M\E$ is a unital subalgebra of $\M$, which is closed in weak operator topology.

Following \cite{Ho, BlLa}, we now consider the following definitions.

\begin{definition}
Let $\A$ be a subalgebra of a $C^*$-algebra $\M$. Then $\A$ is called \emph{logmodular or has logmodularity} in $\M$  if the set $\{a^*a; a\in\A^{-1}\}$ is norm dense in $\M_+^{-1}$. The algebra $\A$ is said to have {\em factorization or strong logmodularity} in $\M$  if  $\{a^*a; a\in\A^{-1}\}=\M_+^{-1}$.
\end{definition}

It is clear that any algebra  having factorization is logmodular. Below we collect some known and straightforward results about logmodular algebras whose proof is simple,  so it is left to the reader. (See Proposition 4.6, \cite{BlLa}).

\begin{proposition}\label{logmodularity are preserved under isomorphism}
Let $\phi:\M\to\mathcal{N}$ be a $*$-isomorphism between two $C^*$-algebras, and let $\A$ be a subalgebra of $\M$. Then $\A$ has logmodularity (resp. factorization) in $\M$ if and only if $\phi(\A)$ has logmodularity (resp. factorization) in $\mathcal{N}$. In particular if $U$ is an appropriate unitary, then  $U^*\A U$ has logmodularity (resp. factorization) in $U^*\M U$ if and only if $\A$ has logmodularity (resp. factorization) in $\M$.
\end{proposition}

\begin{proposition} (Proposition 4.1, \cite{BlLa})
Let $\A$ be a subalgebra of a $C^*$-algebra $\M$. We have the following:
\begin{enumerate}
    \item $\A$ has factorization in $\M$ if and only if $\A^*$ has factorization in $\M$ if and only if for every invertible element $x\in\M$, there exist unitaries $u,v\in\M$ and invertible elements $a,b\in\A^{-1}$ such that $x=ua=bv$.
    \item $\A$ is logmodular in $\M$ if and only if $\A^*$ is logmodular in $\M$ if and only if for each invertible element $x\in\M$, there exist sequences $\{u_{n}\}, \{v_n\}$ of unitaries in $\M$ and invertible elements $\{a_n\}, \{b_n\}$ in $\A^{-1}$ such that $x=\lim_nu_na_n=\lim_nb_nv_n$.
\end{enumerate}
\end{proposition}

There are plenty of such algebras known in literature. The following are examples of logmodular algebras in commutative $C^*$-algebras.

\begin{example}(Function algebras)
A classical result of Szeg\"{o} (Theorem 25.13, \cite{Co}) says that the Hardy algebra $\h^\infty(\T)$ has factorization in $L^\infty(\T,\mu)$. Here $\T$ is the unit circle,  $\mu$ is the one-dimensional Lebesgue measure on $\T$ and $\h^\infty(\T)$ is the algebra of all essentially bounded functions on $\T$ whose negative Fourier coefficients are zero.

More generally, let $m$ be a probability measure, and let $\A$ be a unital subalgebra of $L^\infty(m)$ satisfying the following: (i) $\int fgdm=\int fdm\int gdm$ for all $f,g\in\A$, and $(ii)  $ if $h\in L^1(m)$ with $h\geq0$ a.e. and $\int fhdm=\int fdm$ for all $f\in\A$, then $h=1$ a.e. Let $\h^2(m)$ be the closure of $\A$ in the Hilbert space $L^2(m)$, and let $\h^\infty(m)=\h^2(m)\cap L^\infty(m)$. Then the proof of Theorem 4 in \cite{HoRo} says that $\h^\infty(m)$ has factorization in $L^\infty(m)$. The algebra $\h^\infty(m)$ satisfies many other equivalent conditions analogues to classical Hardy space theory (see Theorem 3.1, \cite{SrWa} for details). Also see \cite{Ho, HoRo, SrWa} for more concrete examples of such measures and algebras.
\end{example}

\begin{example}(Dirichlet algebras)  A closed unital subalgebra $\A$ of a commutative $C^*$-algebra $C(X)$ is called  {\em Dirichlet algebra} if $\A+\bar{\A}$ is uniformly dense in $C(X)$ (equivalently,  $\Re\A$ is uniformly dense in $\Re C(X)$), where $\Re\A$ $[\Re C(X)]$ denotes the set of real parts of the functions in $\A$ $[C(X)]$. If $\A$ is a Dirichlet algebra, then since $\log |\A^{-1}|\subseteq\Re\A$, it is immediate that  $\log |\A^{-1}|$ is dense in $\Re C(X)$; hence $\A$ is a logmodular subalgebra of $C(X)$.  But some Dirichlet algebras may not have factorization. For example, consider the algebra $A(\mathbb{D})$ of all continuous functions on the closed unit disc $\overline{{\mathbb{D}}}$ which is holomorphic on the open unit disc $\mathbb{D}$. Then $A(\mathbb{D})$ is a Dirichlet algebra when considered as the subalgebra of $C(\mathbb{T})$, which is a consequence of Fej\' er-Riesz Theorem on factorization of positive trigonometric polynomials, but $A(\mathbb{D})$ does not have factorization in $C(\T)$. On the other hand, $\h^\infty(\T)$ is an example of an algebra which has factorization, but which is not a Dirichlet algebra. See \cite{ Ho} for details of these facts and more concrete examples of Dirichlet algebras.
\end{example}

To see some examples and other properties of noncommutative algebras having factorization, we recall some  notions to this end. Let $\M$ be a von Neumann algebra, and  let $\E$ be a lattice  of projections in $\M$ (i.e. $\E$ is closed under usual lattice operations $\vee$ and $\wedge$ of finite family of projections). Then  $\E$ is called {\em complete} if $0,1\in\E$, and  $\vee_{i\in\Lambda}p_i$ and $\wedge_{i\in\Lambda}p_i\in\E$ for any arbitrary family $\{p_i\}_{i\in\Lambda}$ in $\E$. The lattice $\E$ is called {\em commutative subspace lattice (CSL)} if projections of $\E$ commute with one another. Moreover,  $\E$ is called  {\em nest} if $\E$ is ordered by usual operator ordering i.e. for any $p,q\in\E$, either $p\leq q$ or $q\leq p$ holds true. We remark here that some authors assume a nest or a CSL to be always complete.

A subalgebra $\A$ is called a {\em nest subalgebra} of $\M$ (or {\em nest algebra} when $\M=\bh$) if $\A=\Alg_\M\E$ for a nest $\E$ in $\M$. Further, a subalgebra $\A$ of  $\M$ is called $\M$-{\em reflexive} (or {\em reflexive} when $\M=\bh$) if $\A=\Alg_\M\Lat_\M\A$. It is clear that any subalgebra of $\M$ of the form $\Alg_\M\E$ for some collection $\E$ of projections in $\M$ is always $\M$-reflexive. In particular,  a nest subalgebra of $\M$ is $\M$-reflexive. It should be noted here that if $\M\subseteq\bh$, then a subalgebra $\A$ of $\M$ can be reflexive in $\bh$, but it need not be $\M$-reflexive.

The following is a well-known result by Larson \cite{La}  regarding factorization property  of nest algebras.

\begin{theorem}\label{Larson's result}(Theorem 4.7, \cite{La})
Let $\E$  be a complete nest in a separable Hilbert space $\h$. Then $\Alg\E$ has factorization in $\bh$ if and only if $\E$ is countable.
\end{theorem}

The following are some examples of algebras having factorization in noncommutative  von Neumann algebras.

\begin{example}(Nest subalgebras) As already mentioned above, $\Alg\E$ has factorization in $\bh$ for any countable complete nest $\E$ in $\bh$. More generally, Pitts proved that if $\E$ is a complete nest in a factor $\M$, then $\Alg_\M\E$ has factorization in $\M$ if and only if ``certain" subnest $\E_r$ of $\E$ is countable (Theorem 6.4, \cite{Pi}).

Moreover, if $\E$ is a nest (not necessarily countable) in a finite von Neumann algebra $\M$ (not necessarily a factor), then $\Alg_\M\E$ has factorization in $\M$ (Corollary 5.11, \cite{Pi}).
\end{example}

\begin{example}(Subdiagonal algebras)  Let $\A$ be a subalgebra of  a  von Neumann algebra  $\M$, and let  $\phi:\M\to\M$ be a faithful normal expectation (i.e. $\phi$ is positive, $\phi(1)=1$ and $\phi\circ\phi=\phi$). Then $\A$ is called a {\em  subdiagonal algebra} with respect to $\phi$ if it satisfies: $(i)$ $\A+\A^*$ is $\sigma$-weakly dense in $\M$, $(ii)$ $\phi(ab)=\phi(a)\phi(b)$ for all $a,b\in\A$, and $(iii)$ $\phi(\A)\subseteq \A\cap\A^*$. Moreover, if the von Neumann algebra $\M$ is finite with a distinguished trace $\tau$, then the subdiagonal algebra $\A$ is called {\em finite} if  $\tau\circ\phi=\tau$. Arveson proved that if $\A$ is a maximal (with respect to $\phi$) finite subdiagonal algebra of $\M$, then $\A$ has factorization in $\M$ (Theorem 4.2.1, \cite{Ar1}). A nest subalgebra of a finite von Neumann algebra  is an example of maximal finite subdiagonal algebras (Corollary 3.1.2, \cite{Ar1}). See  \ref{an example of a subdiagonal algebra} for another concrete example of a finite subdiagonal algebra.

There are other subdiagonal algebras (not necessarily finite) as well, which are known to have factorization. For example, all  subdiagonal algebras arising out of periodic flow have factorization. See \cite{So} for more details of these notions and Corollary 3.11 therein.
\end{example}

We believe that some known facts about  subdiagonal algebras can also be deduced from our result. One such is Theorem 5.1 in \cite{McMuSa}, which follows directly from Corollary \ref{lattices of algebras with factorization}. But we have not explored other possible consequences in depth.

Below we have some concrete examples of nest algebras which do not have factorization.

\begin{example}
Let $\E$ be the nest $\{p_t; t\in [0,1]\}$ of projections on $L^2([0,1])$, where $p_t$ denotes the projection onto $L^2([0,t])$, considered as subspace of $L^2([0,1])$. Then $\E$ is complete and uncountable; hence $\Alg\E$ does not have factorization in $\B(L^2([0,1]))$ by Theorem \ref{Larson's result}. Additionally, let $\mathcal{F}=\{p_i; i\in\mathbb{Q}\}$ be the nest of projections on $\ell^2(\mathbb{Q})$, where $p_i$ denotes the projection onto the subspace $\overline{\Span}\{e_j; j\leq i\}$ for the canonical basis $\{e_i; i\in\mathbb{Q}\}$ of $\ell^2(\mathbb{Q})$. Although $\mathcal{F}$ is a countable nest, it is easy to verify that its completion is not countable (actually indexed by $\R\sqcup\mathbb{Q}$) and hence $\Alg\mathcal{F}$ does not have factorization in $\B(\ell^2(\mathbb{Q}))$. At this point, we do not know whether these algebras are logmodular.
\end{example}

We now state the main result of this paper. This gives us the understanding of behaviour of the lattices of  logmodular algebras.

\begin{theorem}\label{lattices of logmodular algebras}
Let $\A$ be a logmodular algebra in a von Neumann algebra $\M$. Then $\Lat_\M\A$ is a commutative subspace lattice. Moreover if $\M$  is a factor, then $\Lat_\M\A$ is a nest.
\end{theorem}

We postpone the proof of Theorem \ref{lattices of logmodular algebras} to the next section, and instead look at some of its consequences first. Note that since any algebra having factorization  is also logmodular, the following corollary is immediate.

\begin{corollary}\label{lattices of algebras with factorization}
Let an algebra $\A$ have factorization   in a von Neumann algebra $\M$. Then $\Lat_\M\A$ is a commutative subspace lattice. Moreover if $\M$ is a factor, then $\Lat_\M\A$ is a nest.
\end{corollary}

\begin{remark}
If $\M$ is an arbitrary von Neumann algebra which is not a factor, and $\A$ is  a subalgebra of $\M$, then we can never expect $\Lat_\M\A$ to be a nest irrespective of whether $\A$ is logmodular or has factorization.  In fact if $\mathcal{P}_{\mathcal{Z}}$  denotes the lattice of all projections in the center $\mathcal{Z}$ of $\M$, then it is always true that $\mathcal{P}_{\mathcal{Z}}\subseteq \Lat_\M\A$.  So $\Lat_\M\A$ can never be a nest if the center $\mathcal{Z}$ is non-trivial.
\end{remark}

Now let $\M$ be a factor, and let $\A$ be an $\M$-reflexive subalgebra of $\M$. If $\A$ is logmodular in $\M$, then $\Lat_\M\A$ is a nest by Theorem \ref{lattices of logmodular algebras}. But since $\A=\Alg_\M\Lat_\M\A$, it follows that $\A$ is a nest subalgebra of $\M$.

We now answer an open question posed by Paulsen and Raghupathi (see pg. 2630,  \cite{PaRa}) using above observation. They conjectured that every completely distributive CSL  logmodular algebra in $\bh$ is a nest algebra. Here a (completely distributive) CSL algebra means an algebra of the form $\Alg\E$, where $\E$ is a (completely distributive) commutative subspace lattice (see \cite{Da} for more on completely distributive CSL algebras). Note that all nests are completely distributive. Since any CSL algebra is a special case of reflexive algebras, we have thus answered  their question  in affirmative. We record it below.

\begin{corollary}
An $\M$-reflexive logmodular algebra in a factor $\M$ is a nest subalgebra of $\M$. In particular, all reflexive (hence completely distributive CSL) logmodular algebras in $\bh$ are nest algebras.
\end{corollary}

If an algebra $\A$ has factorization in $\bh$, then $\Alg\Lat\A$ also has factorization in $\bh$ as $\A$ is contained in $\Alg\Lat\A$. Since $\Lat\A$ is a complete nest,  it then follows from Theorem \ref{Larson's result} that $\Lat\A$ is a countable nest. In particular, if $\A=\Alg\E$ for a  lattice $\E$ of projections in $\h$, then $\E$ is a countable nest because $\E\subseteq\Lat\A$. Thus we get the following corollary, which  is  a  strengthening of Theorem \ref{Larson's result} of Larson.

\begin{corollary}\label{strenghtening of Larson result}
Let $\E$ be a complete lattice of projections on a separable  Hilbert space $\h$. Then $\Alg\E$ has factorization in $\bh$ if and only if $\E$ is a countable nest.
\end{corollary}

\section{Proof of the main result}

This section is devoted to the proof of our main result (Theorem \ref{lattices of logmodular algebras}). We first discuss some general ingredients required for this. A simple observation that we shall be using throughout the article is the following remark. Recall that  $p^\perp$ denotes the projection $1-p$ for any projection $p$.

\begin{remark}
For any subalgebra $\A$ of a von Neumann algebra $\M$,  $p\in\Lat_\M\A \iff ap=pap\;\; \forall a\in\A\iff pa^*=pa^*p\; \forall a\in\A\iff a^*p^\perp=p^\perp a^* p^\perp\;\;\forall a\in\A\iff  p^\perp\in\Lat_\M\A^*$.
\end{remark}

The following proposition says that logmodularity and factorization are preserved under  compression of algebras by appropriates projections. Here $p\A p$ denotes the subspace $\{pap; a\in\A\}$ for any projection $p$ and an algebra $\A$. Note that $p\A p$ need not always be an algebra.

\begin{proposition}\label{logmodularity  is preserved on compression}
Let an algebra $\A$ have logmodularity (resp. factorization) in a von Neumann algebra $\M$, and let $p,q\in \Lat_\M\A$  be such that $p\geq q$. Then  the following statements are true:
\begin{enumerate}
    \item $p\A p   (=\A p)$ has  logmodularity (resp. factorization)  in  $p\M p$.
    \item $p^\perp\A p^\perp$ has  logmodularity (resp. factorization) in $p^\perp\M p^\perp$.
    \item  $(p-q)\A (p-q)$ has  logmodularity (resp. factorization)  in $(p-q)\M (p-q)$.
\end{enumerate}
\end{proposition}
\begin{proof}
We shall prove only part (3). Part (1)  follows from $(3)$ by taking $q=0$,  and (2) follows from $(3)$ by taking $p=1$ and $q=p$. Also we shall prove only the case of logmodularity. That of factorization follows similarly. So assume that $\A$ is logmodular in $\M$.

First we show that $(p-q)\A(p-q)$ is an algebra. For all $a\in\A$, since $ap=pap$ and $aq=qaq$, we note that
\begin{align*}
    (p-q)aq=(p-q)qaq=0,
\end{align*}
and
\begin{align}\label{pa(p-q)=a(p-q)}
    pa(p-q)=pap-paq=ap-pqaq=ap-qaq=ap-aq=a(p-q).
\end{align}
Combining the two expressions above,  it follows for all $a,b\in\A$ that
\begin{align}\label{eq: eqation give algebra}
    (p-q)a(p-q)b(p-q)=(p-q)apb(p-q)-(p-q)aqb(p-q)=(p-q)ab(p-q),
\end{align}
which shows that $(p-q)\A(p-q)$ is an algebra.

To show that $(p-q)\A(p-q)$ is logmodular in $(p-q)\M(p-q)$,  fix  a positive and invertible element $x$ in $(p-q)\M(p-q)$ and set $\tilde{x}=x+q+p^\perp$. Note that $x=(p-q)\tilde{x}(p-q)$. It is clear that $\tilde{x}$ is positive in $\M$. Since $x$ is positive and invertible in $(p-q)\M(p-q)$, there is some $\alpha\in (0,1)$ such that $x\geq \alpha(p-q)$, from which we get
\begin{align*}
    \tilde{x}=x+q+p^\perp\geq \alpha (p-q)+\alpha q+\alpha p^\perp=\alpha.
\end{align*}
This shows that $\tilde{x}$ is invertible in $\M$. We then use logmodularity of $\A$ in $\M$ to get a sequence $\{\tilde{a}_n\}$ in $\A^{-1}$ such that $\tilde{x}=\lim_n\tilde{a}_n^*\tilde{a}_n$. So for each $n$, we have $\tilde{a}_nq=q\tilde{a}_nq$ and $\tilde{a}_n^{-1}q=q\tilde{a}_n^{-1}q$. It then follows that
\begin{align*}
    (q\tilde{a}_nq)(q\tilde{a}_n^{-1}q)=q\tilde{a}_n\tilde{a}_n^{-1}q=q~~\mbox{ and }~~(q\tilde{a}_n^{-1}q)(q\tilde{a}_nq)=q\tilde{a}_n^{-1}\tilde{a}_nq=q,
\end{align*}
which is to say that $q\tilde{a}_nq$  is invertible in $q\M q$ with $(q\tilde{a}_nq)^{-1} =q\tilde{a}_n^{-1}q\in q\A q$. In particular, since the sequence $\{\tilde{a}_n^{-1}\}$ is bounded (as  $\{(\tilde{a}_n^*\tilde{a}_n)^{-1}\} $ is a convergent sequence), it follows that the sequence $\{(q\tilde{a}_nq)^{-1}\}$ is bounded. Note that $q\tilde{x}(p-q)=0$, and since $q\tilde{a}_n^*=q\tilde{a}_n^*q$ for all $n$, we have
\begin{align*}
    0=q\tilde{x}(p-q)=\lim_n q\tilde{a}_n^*\tilde{a}_n(p-q)=\lim_n(q\tilde{a}_n^*q)(q\tilde{a}_n(p-q)).
\end{align*}
Multiplying to the left of the sequence  by $(q\tilde{a}_n^*q)^{-1}$ yields
\begin{align*}
    \lim_nq \tilde{a}_n(p-q)=0,
\end{align*}
using which and the expression $\tilde{a}_n(p-q)=p\tilde{a}_n(p-q)$ from \eqref{pa(p-q)=a(p-q)}, we get the following:
\begin{align*}
    x&=(p-q)\tilde{x}(p-q)=\lim_n(p-q)\tilde{a}_n^*\tilde{a}_n(p-q)=\lim_n(p-q)\tilde{a}_n^*p\tilde{a}_n(p-q)\\
    &=\lim_n(p-q)\tilde{a}_n^*[q \tilde{a}_n(p-q)]+\lim_n(p-q)\tilde{a}_n^*[(p-q)\tilde{a}_n(p-q)]\\
    &=\lim_n(p-q)\tilde{a}_n^*(p-q)\tilde{a}_n(p-q)=\lim_na_n^*a_n,
\end{align*}
where $a_n=(p-q)\tilde{a}_n(p-q)\in(p-q)\A(p-q)$.
Also for each $n$, we have from \eqref{eq: eqation give algebra} that
\begin{align*}
    p-q= (p-q)\tilde{a}_n^{-1}(p-q)\tilde{a}_n(p-q)=(p-q)\tilde{a}_n(p-q)\tilde{a}_n^{-1}(p-q),
\end{align*}
which shows that $a_n=(p-q)\tilde{a}_n(p-q)$ is invertible with inverse $(p-q)\tilde{a}_n^{-1}(p-q)$ in $(p-q)\A(p-q)$. Thus we get a sequence $\{a_n\}$ of invertible elements with $a_n,a_n^{-1}\in (p-q)\A(p-q)$ for all $n$ such that $x=\lim_na_n^*a_n$. Since $x$ is an arbitrary positive and invertible element, we conclude  that $(p-q)\A(p-q)$ is logmodular in $(p-q)\M(p-q)$. 
\end{proof}

We now recall some basic facts about subspaces in a separable Hilbert space. Following Halmos \cite{Ha}, we say two non-zero subspaces $E$ and $F$ of a Hilbert space are  in {\em generic position} if all the following subspaces
\begin{align*}
    E\cap F, ~E\cap F^\perp, E^\perp\cap F~\mbox{and}~ E^\perp\cap F^\perp
\end{align*}
are zero. We are going to use the following characterization of subspaces in generic position. Recall that $P_E$ denotes the projection onto a subspace $E$. Also recall that $\ker x$ denotes the kernel of any operator $x$.

\begin{lemma}\label{Structure of subspaces in generic position}(Theorem 2, \cite{Ha})
Let $E$ and $F$ be two subspaces in generic position in a Hilbert space $\h$. Then there exist a Hilbert space $\K$, a unitary  $U:\h\to\K\oplus\K$, and commuting positive contractions $x,y\in\bk$ such that $x^2+y^2=1$ and $\ker x=\ker y=0$ and
\begin{align*}
    UP_EU^*=\begin{bmatrix}
    1&0\\
    0&0
    \end{bmatrix}~~\mbox{and}~~UP_FU^*=\begin{bmatrix}
    x^2&xy\\
    xy&y^2
    \end{bmatrix}.
\end{align*}
\end{lemma}

The following lemma is very easy to verify, whose proof is left to the readers.

\begin{lemma}\label{M_1 and N_1 are generic part}
Let $E$ and $F$ be two subspaces in a Hilbert space $\h$, and let $\h_1$ denote the subspace of $\h$ given by
\begin{align*}
\h_1=\h\ominus\left(E\cap F+E\cap F^\perp+E^\perp\cap F+E^\perp\cap F^\perp\right).
\end{align*}
If   $E_1=E\cap\h_1$ and $F_1=F\cap\h_1$, then exactly one of the following holds true:
\begin{enumerate}
    \item $\h_1=\{ 0\}$, and hence $E_1, F_1=0$.
    \item $E_1$ and $F_1$ are non-zero, and they are in generic position as subspaces of $\h_1$.
\end{enumerate}
Moreover, the projections $P_E$ and $P_F$ commute if and only if the first condition is satisfied (i.e. $\h_1=0$).
\end{lemma}

The subspaces $E_1$ and $F_1$ as in  Lemma 3.4 are called {\em generic part} of the subspaces $E$ and $F$. The structure of general subspaces can now be described using the two lemmas above.

\begin{proposition}\label{structure of any two subspaces}
Let $E$ and $F$ be two subspaces in a Hilbert space $\h$. Then there is a Hilbert space $\K$ (could be zero), and commuting positive contractions $x,y\in\bk$ with $x^2+y^2=1, \ker x=\ker y=0$ such that, upto unitary equivalence 
$$\h=E\cap F\oplus E\cap F^\perp\oplus E^\perp\cap F\oplus E^\perp\cap F^\perp\oplus\K\oplus\K,$$
and
\begin{align*}
   & P_E= 1\oplus 1\oplus 0\oplus0\oplus 1\oplus0~\mbox{ and }~
    P_F=1\oplus0\oplus1\oplus0\oplus \begin{bmatrix}
    x^2&xy\\
    xy&y^2
    \end{bmatrix}.
\end{align*}
Here any of the components in the decomposition could be $0$. Moreover, $P_EP_F=P_FP_E=P_{E\cap F}$ if and only if $\K=\{0\}$.
\end{proposition}

We are now ready to give proof of our main result through a series of lemmas. The next two lemmas deal with factor von Neumann algebras only, where we use the following comparison theorem of projections  (see Corollary 47.9, \cite{Co}): If $\M$ is a factor and $p,q$ are two non-zero projections in $\M$, then either $p\preceq q$ or $q\preceq p$  i.e. there is a non-zero partial isometry $v\in\M$ such that $v^*v\leq p$ and $vv^*\leq q$. The same is clearly not true for arbitrary von Neumann algebras.

\begin{lemma}\label{subspaces with logmodularity properties cannot be orthogonal}
Let $\M$ be a factor, and let $p,q$ be mutually orthogonal projections in $\M$. Then $\Alg_\M\{p,q\}$ is not logmodular in $\M$.
\end{lemma}
\begin{proof}
Since $\M$  is a factor and $p,q\in\M$ are non-zero, it follows that there is a non-zero partial isometry $v\in\M$ such that $v^*v\leq p$ and $vv^*\leq q$.  In particular, we have $v=qv=vp$, and since $pq=0$ it follows that $pv=pqv=0$.

Assume to the contrary that $\Alg_\M\{p,q\}$ is logmodular in $\M$. Let $x=1+\epsilon(v+v^*)$ for some $\epsilon>0$, where $\epsilon$ is chosen small enough so that $x$ is positive and invertible in $\M$. Then there exists a sequence $\{a_n\}$ of invertible elements in $\Alg_\M\{p,q\}$ such that $x=\lim_na_n^*a_n$. Now since $v^*p=0$ and $pq=0$, we note that $qxp=\epsilon qvp=\epsilon v$. We also have $a_np=pa_np$ and $qa_n^*=qa_n^*q$ for all $n$;  thus we get
\begin{align*}
    \epsilon v= qxp=\lim_nqa_n^*a_np=\lim_nqa_n^*qpa_np=0,
\end{align*}
which is a contradiction, as  $v\neq0$.
\end{proof}

We recall here a simple fact that if $p$ and $q$ are commuting projections, then $pq$ is a projection such that $p\wedge q=pq$ and $p\vee q=p+q-pq$.

\begin{lemma}\label{the main result for factors}
Let $\M$  be a factor, and let $p,q\in\M$ be two commuting projections. If $\Alg_\M\{p.q\}$ is logmodular  in $\M$, then either $p\leq q$ or $q\leq p$ holds true.
\end{lemma}
\begin{proof}
Since $p$ and $q$ commuting projections, it follows that the operators $pq, pq^\perp$ and  $p^\perp q$ are projections. From Lemma \ref{subspaces with logmodularity properties cannot be orthogonal}, we know that $pq\neq 0$. Note that the required assertion will follow by the following argument once we show that either $pq^\perp=0$ or $p^\perp q=0$: say $pq^\perp=0$, then $p=p(q+q^\perp)=pq$ which implies that $p\leq q$.  Similarly, $p^\perp q=0$ will imply $q\leq p$.

Assume to the contrary that both the projections $pq^\perp$ and $ p^\perp q$ are non-zero. Since $\M$  is a factor, there is a non-zero partial isometry $v\in\M$ such that $v^*v\leq pq^\perp $ and $vv^*\leq p^\perp q$; in particular we have,
\begin{align}\label{v=vp perp}
    v=vpq^\perp=p^\perp qv.
\end{align}
Now let $x=1+\epsilon(v+v^*)$ for $\epsilon>0$, where we choose $\epsilon$ small enough so that $x$ is positive and invertible in $\M$. Since $\Alg_\M\{p,q\}$ is logmodular in $\M$, there exists a sequence $\{a_n\}$ of invertible elements in $\M$  such that $a_n, a_n^{-1}\in\Alg_\M\{p,q\}$ for all $n$ and $x=\lim_na_n^*a_n$. Note that $pqa_npq=a_npq$ and $pqa_n^{-1}pq=a_n^{-1}pq$ for all $n$; hence $pqa_npq$ is invertible in $pq\M pq$ with inverse $pqa_n^{-1}pq$. In particular, the sequence $\{(pqa_npq)^{-1}\}$ is bounded, as the sequence $\{a_n^{-1}\}$ is bounded. Also note from $\eqref{v=vp perp}$ that $vpq=0$ and $p^\perp qv^*=0$; hence we get
$$p^\perp qxpq=p^\perp qpq+\epsilon p^\perp q(vpq)+\epsilon(p^\perp qv^*)pq=0.$$
 Thus we have
\begin{align*}
    0=p^\perp qxpq=\lim_np^\perp qa_n^*a_npq=\lim_np^\perp qa_n^*pqa_npq,
\end{align*}
and so by multiplying by $\{(pqa_npq)^{-1}\}$ to  right side of the sequence, we get $\lim_np^\perp qa_n^*pq=0$; using which  and the expressions $q a_n^*=qa_n^*q$ and $a_np=pa_np$ for all $n$, it follows that
\begin{align*}
    p^\perp qxpq^\perp=\lim_np^\perp qa_n^*a_npq^\perp
    =\lim_n(p^\perp qa_n^*qp)a_npq^\perp=0.
\end{align*}
On the other hand, we again use the equation $p^\perp qv^*=0$ to get
$$p^\perp qxpq^\perp= p^\perp qpq^\perp+\epsilon p^\perp q vpq^\perp+\epsilon p^\perp qv^*pq^\perp=\epsilon p^\perp qvpq^\perp=\epsilon v\neq 0.$$
So we get a contradiction, which arose because we assumed that both $pq^\perp$ and $p^\perp q$ are non-zero. Thus one of them is zero and we have the required result.
\end{proof}

We are going to use the following simple lemma very frequently.

\begin{lemma}\label{in limits, invertible isometries behave like unitaries}
Let $\{a_n\}$ be a sequence of invertible elements in a $C^*$-algebra $\M$  such that $\lim_na_n^*a_n=1$. Then $\{a_n^{-1}\}$ is bounded and $\lim_na_na_n^*=1$.
\end{lemma}
\begin{proof}
Since $\lim_na_n^*a_n=1$, it follows that $\lim_n(a_n^*a_n)^{-1}=1$ and so $\{(a_n^*a_n)^{-1}\}$ is bounded. This implies the first assertion that $\{a_n^{-1}\}$ is bounded. Further we have $\lim_na_n^*a_na_n^*a_n=1$, and  hence
\begin{align*}
    0=\lim_n( a_n^*a_na_n^*a_n-a_n^*a_n)=\lim_na_n^*(a_na_n^*-1)a_n.
\end{align*}
Since the sequence $\{a_n^{-1}\}$ is bounded, it follows by multiplying ${a_n^*}^{-1}$ to the left and $a_n^{-1}$ to the right of the sequence   that $\lim_n(a_na_n^*-1)=0, $ as to be proved.
\end{proof}

Next we consider lattices of logmodular algebras in arbitrary von Neumann algebras, where our aim is to prove that the generic part of any two invariant subspaces is zero. Recall that $\ran(x)$ denotes the range of an operator $x$.

\begin{lemma}\label{subspaces with logmodularity properties cannot be in generic position}
Let $\M$ be a von Neumann subalgebra of $\bh$ for some Hilbert $\h$, and let $p,q$ be two non-zero projections in $\M$ such that  $\ran(p)$ and $\ran(q)$ are  in generic position in $\h$. Then $\Alg_\M\{p,q\}$ is not logmodular in $\M$.
\end{lemma}
\begin{proof}
Assume to the contrary that the algebra $\Alg_\M\{p,q\}$ is logmodular in $\M$. Since $\ran(p)$ and $\ran(q)$ are in generic position in $\h$, it follows from  Lemma \ref{Structure of subspaces in generic position} that there exist a Hilbert space $\K$  and commuting positive contractions $x,y\in\bk$ satisfying $\ker x=0,$ $\ker y=0$ and $x^2+y^2=1$  such that upto unitary equivalence, we have $\h=\K\oplus\K$ and
\begin{align}\label{for for projections p,q}
    p=\begin{bmatrix}
        1&0\\
        0&0
    \end{bmatrix}~~\mbox{and}~~
    q=\begin{bmatrix}
        x^2&xy\\
        xy&y^2
    \end{bmatrix}.
\end{align}
Since logmodularity is preserved under unitary equivalence by Proposition \ref{logmodularity are preserved under isomorphism}, we can assume without loss of generality that $\M$ is a von Neumann subalgebra of $\B(\K\oplus\K)$, and $p,q$ are of the form as in $\eqref{for for projections p,q}$.

Now let $S$ be an invertible operator  such that $S,S^{-1}\in\Alg_\M\{p,q\}$. Then $Sp=pSp$ and  $S^{-1}p=pS^{-1}p$, which imply that $S$ and $S^{-1}$ have the following form:
\begin{align*}
    S=\begin{bmatrix}
        a&b\\
        0&c
    \end{bmatrix}~~\mbox{and}~~
    S^{-1}=\begin{bmatrix}
        a'&b'\\
        0&c'
    \end{bmatrix},
\end{align*}
for some operators $a,b,c,a',b',c'\in\bk$. It is then clear from the expression $SS^{-1}=1=S^{-1}S$ that $a$ and $c$ are invertible in $\bk$ with respective inverses $a'$ and $c'$. Note that
\begin{align*}
    Sq=\begin{bmatrix}
        ax^2+bxy&axy+by^2\\
        cxy&cy^2
    \end{bmatrix}~\mbox{and}~
    qSq=\begin{bmatrix}
       x^2ax^2+x^2bxy+xycxy&x^2axy+x^2by^2+xycy^2\\
        xyax^2+xybxy+y^2cxy&xyaxy+xyby^2+y^2cy^2
    \end{bmatrix}.
\end{align*}
Since $Sq=qSq$, we equate  $(2,1)$ entries of the matrices $Sq$ and $qSq$, and use the equation $1-y^2=x^2$  to get the expression $x^2cxy=xyax^2+xybxy.$ Since $x$ is injective (and hence $x$ has dense range, as $x$ is positive), $x$ can be cancelled from both the sides to get the following:
\begin{align}\label{eq: relation between operators a,b,c}
   xcy=yax+yby.
\end{align}
Now fix $\alpha \geq1$, and let $Z=\begin{bmatrix}
    1&\alpha\\
    \alpha&\alpha^2+1
\end{bmatrix}\in\B(\K\oplus\K)$. It is clear that $Z$ is a positive and invertible operator. We claim that $Z\in \M$. Since $p$ and $q$ are in $\M$, it follows that
$\begin{bmatrix}
x^2&0\\
0&0\end{bmatrix}=pqp\in\M$.
Similarly
$\begin{bmatrix}0&0\\
0&y^2
\end{bmatrix}=p^\perp qp^\perp\in\M$.
Thus
$\begin{bmatrix}x^2&0\\
0&y^2
\end{bmatrix}\in\M$
and hence
$\begin{bmatrix}0&xy\\
xy&0
\end{bmatrix}\in\M$.
Set
$T=\begin{bmatrix}0&xy\\
xy&0
\end{bmatrix}$
and let $T=U|T|$ be its polar decomposition, where $|T|$ denotes the square root of the operator $T^*T$. It is clear that $T$ is one-one (as $xy$ is one-one), and so $U$ is unitary. It is straightforward to check (using uniqueness of polar decomposition) that
$|T|=\begin{bmatrix}xy&0\\
0&xy
\end{bmatrix}$
and $U=\begin{bmatrix}0&1\\
1&0
\end{bmatrix}$.
Since $\M$ is a von Neumann algebra,   it follows that $U\in\M$ and so
$\begin{bmatrix}0&\alpha\\
\alpha&0
\end{bmatrix}=\alpha U\in\M$.
Also since
$\begin{bmatrix}1&0\\
0&\alpha^2+1
\end{bmatrix}=p+(\alpha^2+1)p^\perp\in\M$, we conclude that $Z\in\M$, as claimed.

Thus by logmodularity  of $\Alg_\M\{p,q\}$ in $\M$, we get a sequence $\{S_n\}$ of invertible operators with $S_n,S_n^{-1}\in\Alg_\M\{p,q\}$ for all $n$ such that $Z=\lim_nS_n^*S_n$. It then follows from above discussion that for each $n$, $S_n$ is of the form
$$S_n=\begin{bmatrix}
    a_n&b_n\\
    0&c_n
\end{bmatrix},$$
for $a_n,b_n,c_n\in\bk$ such that $a_n$ and $c_n$ are invertible operators, and from  \eqref{eq: relation between operators a,b,c} we have
\begin{align}\label{relation between a_n,b_n,c_n}
    xc_ny=ya_nx+yb_ny.
\end{align}
Now we have
\begin{align}\label{expressions for Z_n and S_n*S_n}
    \begin{bmatrix}1&\alpha\\
    \alpha&\alpha^2+1
    \end{bmatrix}=Z=\lim_nS_n^*S_n=
    \lim_n\begin{bmatrix}
        a_n^*a_n&a_n^*b_n\\
        b_n^*a_n&b_n^*b_n+c_n^*c_n
    \end{bmatrix}.
\end{align}
So we get $\lim_na^*_na_n=1$, and since each $a_n$ is invertible, it follows from Lemma \ref{in limits, invertible isometries behave like unitaries} that
\begin{align}\label{lim_na_na_n*=1}
    \lim_na_na_n^*=1.
\end{align}
We also get from \eqref{expressions for Z_n and S_n*S_n} that $\lim_na_n^*b_n=\alpha$, which further yields by multiplying $a_n$ to the left side of the sequence and using \eqref{lim_na_na_n*=1} that
\begin{align}\label{lim_nb_n-alpha a_n=0}
   \lim_n(b_n-\alpha a_n)=0.
\end{align}
Set  $d_n=b_n-\alpha a_n$ for all $n$. Then $\lim_nd_n=0$, and since $\lim_na_n^*a_n=1$  we have
\begin{align*}
    \lim_nb_n^*b_n=\lim_n(d_n+\alpha a_n)^*(d_n+\alpha a_n)=\lim_n\alpha^2 a_n^*a_n=\alpha^2.
\end{align*}
We further get from \eqref{expressions for Z_n and S_n*S_n} that $\alpha^2+1=\lim_n(b_n^*b_n+c_n^*c_n)$, which
yields $\lim_nc_n^*c_n=1$.
Again as each $c_n$ is invertible, it follows from Lemma \ref{in limits, invertible isometries behave like unitaries} that
\begin{align}\label{lim_nc_nc_n*=1}
    \lim_nc_nc_n^*=1.
\end{align}
Now from \eqref{relation between a_n,b_n,c_n} and using $b_n=\alpha a_n+d_n$, we have
\begin{align*}
     xc_ny=ya_nx+yb_ny=ya_nx+\alpha  ya_ny+yd_ny=ya_n(x+\alpha  y)+yd_ny=ya_nu+yd_ny,
\end{align*}
where $u=x+\alpha y$. Since $\alpha\geq1$, we note that $u$ is positive and invertible (in fact $u^2=1+(\alpha^2-1)y^2+2\alpha xy\geq 1$),
and thus we get
\begin{align}\label{eq:xc_nyu-1}
    ya_n=xc_nyu^{-1}-yd_nyu^{-1}.
\end{align}
Note that
\begin{align*}
    u^2&=(x+\alpha  y)^2=x^2+\alpha^2y^2+2\alpha xy\geq \alpha^2y^2
\end{align*}
and since $y$ and $u$ commutes, it follows that
\begin{align}\label{eq:y^2u^{-2}}
    y^2u^{-2}\leq 1/{\alpha^2}.
\end{align}
Thus we use the expression $\lim_nd_n=0$ from \eqref{lim_nb_n-alpha a_n=0}, and equations in \eqref{lim_na_na_n*=1}, $\eqref{lim_nc_nc_n*=1}$,   \eqref{eq:xc_nyu-1} and \eqref{eq:y^2u^{-2}} to get the following:
\begin{align*}
    y^2&=\lim_nya_na_n^*y=\lim_n(ya_n)(ya_n)^*
=\lim_n(xc_nyu^{-1}-yd_nyu^{-1})(xc_nyu^{-1}-yd_nyu^{-1})^*\\
    &=\lim_n(xc_nyu^{-1})(xc_nyu^{-1})^*=\lim_nxc_ny^2u^{-2}c_n^{*}x
    \leq\lim_n\frac{1}{\alpha^2}xc_nc_n^*x= \frac{1}{\alpha^2}x^2.
\end{align*}
Since $\alpha\geq 1$ is arbitrary, it follows by letting $\alpha$ tend to $\infty$  that $y=0$, which is clearly not true. Thus our assumption that $\Alg_\M\{p,q\}$ is logmodular is false completing the proof.
\end{proof}

Finally we  prove our main theorem in full generality, for which we consider the following lemma.

\begin{lemma}\label{M_2 and N_2 has the logmodularity properties}
Let an algebra $\A$ have logmodularity (resp. factorization) in a von Neumann algebra $\M$, and let $p,q\in \Lat_\M\A$. If $r=(p\wedge q)\vee(p^\perp\wedge q^\perp)$, then $r^\perp\A r^\perp$ has logmodularity (resp. factorization) in  $r^\perp\M r^\perp$.
\end{lemma}
\begin{proof}
Set $r_1=p\wedge q$ and $r_2=p^\perp\wedge q^\perp$. It is clear that $r_1r_2=0$ and $r=r_1+r_2$. Since $p,q\in\Lat_\M\A$, it follows that $r_1\in\Lat_\M\A$. Also we note that $p^\perp, q^\perp\in\Lat_\M\A^*$ and hence $r_2\in\Lat_\M\A^*$, which is to say that $r_2^\perp\in\Lat_\M\A$. Note that $r^\perp=1-r_2-r_1=r_2^\perp-r_1$, and so $r_1\leq r_2^\perp$. Both the assertions about logmodularity and factorization now follow from part (3) of Proposition \ref{logmodularity  is preserved on compression}.
\end{proof}

\noindent
{\em Proof of Theorem \ref{lattices of logmodular algebras}.} Let $\A$ be a logmodular subalgebra of a von Neumann algebra $\M$, and let $p,q\in\Lat_\M\A$. We have to show that $pq=qp$.   The second assertion that $p\leq q$ or $q\leq p$  whenever $\M$ is a  factor,  will then follow from Lemma \ref{the main result for factors}.

Set $r=(p\wedge q)\vee (p^\perp\wedge q^\perp)$. Then $r^\perp\A r^\perp$ is a logmodular algebra in $r^\perp\M r^\perp$ by Lemma \ref{M_2 and N_2 has the logmodularity properties}. Note that the projections $p$ and $q$ commute with $r$, and hence  with $r^\perp$. So if we set  $p'= r^\perp pr^\perp $ and $q'= r^\perp qr^\perp $, then it is immediate that $p',q'$ are projections in $r^\perp\M r^\perp$, and we have $p'=p\wedge r^\perp $ and $q'=q\wedge r^\perp$.  Note that $pq(p\wedge q)=p\wedge q=qp(p\wedge q)$ and $pq(p^\perp\wedge q^\perp)=0=qp(p^\perp \wedge q^\perp)$; hence $pqr=p\wedge q=qpr$, which further yields
\begin{align*}
   &pq=pq(r+r^\perp)=pqr+pqr^\perp=p\wedge q+(r^\perp pr^\perp)(r^\perp qr^\perp)=p\wedge q+p'q',\\
  &qp=qpr+qpr^\perp=p\wedge q+(r^\perp qr^\perp)(r^\perp pr^\perp)=p\wedge q+q'p'.
\end{align*}
Therefore, in order to show the required assertion it is enough to prove that $p'q'=q'p'$. Also we note that
$$p'\wedge q'=p\wedge q\wedge r^\perp\leq r\wedge r^\perp=0,$$
and
$$(r^\perp-p')\wedge(r^\perp-q')=(r^\perp-pr^\perp)\wedge(r^\perp-qr^\perp)=p^\perp r^\perp\wedge q^\perp r^\perp
=(p^\perp\wedge q^\perp)\wedge r^\perp\leq r\wedge r^\perp=0.$$
Here $r^\perp-p'$ and $r^\perp-q'$ are the orthogonal complement of the projections $p'$ and $q'$  in  $r^\perp\M r^\perp$ respectively.
Thus if necessary, by replacing the algebras $\M$ and $\A$ by $r^\perp\M r^\perp$ and $r^\perp\A r^\perp$ respectively, and the projections $p,q$ by $p',q'$ respectively we assume without loss of generality that
\begin{align}
    p\wedge q=0=p^\perp\wedge q^\perp,
\end{align}
so that  $r=0$ and  $\M=r^\perp\M r^\perp$.
The purpose of reducing $\M$ to $r^\perp\M r^\perp$ is just to avoid multiple cases, and work with $4\times 4$ matrices rather than $6\times6$ matrices as we shall see.

Now assume to the contrary that $pq\neq qp$. Then the generic part of $\ran( p)$ and $\ran(q)$ in $\h$ are non-zero by  Proposition \ref{structure of any two subspaces}, where $\h$ is the Hilbert space on which the von Neumann algebra $\M$ acts. Further if both $p\wedge q^\perp $ and $p^\perp\wedge q$ are also zero, then (as $p\wedge q=0=p^\perp\wedge q^\perp$) $\ran(p)$ and $\ran(q)$ will be in generic position, which is not possible by   Lemma \ref{subspaces with logmodularity properties cannot be in generic position}  since $\Alg_\M\{p,q\}$ is logmodular in $\M$ as well. Therefore at least one of the projections $p\wedge q^\perp$ and $p^\perp\wedge q$ is non-zero. Thus for the remainder of the proof,   we  assume that both $p^\perp\wedge q$ and $p\wedge q^\perp$ are non-zero (the case of exactly one of them being non-zero follows similarly). It then follows from Proposition \ref{structure of any two subspaces} that there exist a non-zero Hilbert space $\K$ and commuting positive contractions $x,y\in\bk$ satisfying $x^2+y^2=1$ and $\ker x=0=\ker y$ such that upto unitary unitary equivalence, we have
\begin{align}\label{decomposition of H}
    \h=\ran(p\wedge q^\perp)\oplus\K\oplus\K\oplus\ran(p^\perp\wedge q)
\end{align}
and
\begin{align}\label{how P_M and P_N looks like}
    p=\begin{bmatrix}
        1&0&0&0\\
        0&1&0&0\\
        0&0&0&0\\
        0&0&0&0
    \end{bmatrix}~~\mbox{and}~~
    q=\begin{bmatrix}
        0&0&0&0\\
        0&x^2&xy&0\\
        0&xy&y^2&0\\
        0&0&0&1
    \end{bmatrix}.
\end{align}
Since logmodularity is preserved under unitary equivalence by Proposition \ref{logmodularity are preserved under isomorphism}, we assume without loss of generality that $\M$ is a von Neumann subalgebra of $\B(\ran(p\wedge q^\perp)\oplus\K\oplus\K\oplus\ran(p^\perp\wedge q))$, and $p,q$ have the form as in \eqref{how P_M and P_N looks like}. Now set
$$\widetilde{\K}_1=\ran(p\wedge q^\perp)\oplus\K~~ \mbox{and} ~~\widetilde{\K}_2=\K\oplus \ran(p^\perp \wedge q)$$
so that
\begin{align}\label{decomposition of H 2}
    \h=\widetilde{\K}_1\oplus\widetilde{\K}_2.
\end{align}
Throughout the proof, we make use of both the decomposition of $\h$ in \eqref{decomposition of H} and \eqref{decomposition of H 2}, which should be understood according to the context. Now fix $\alpha  \geq1 $ and define the operator $Z\in\bh$ by
\begin{align}
    Z=\begin{bmatrix}
        1&0&0&0\\
        0&1&\alpha&0\\
        0&\alpha&\alpha^2+1&0\\
        0&0&0&1
    \end{bmatrix}
    =\begin{bmatrix}
        1& Z_2\\
        Z_2^*&Z_3
    \end{bmatrix},~\mbox{where }
    Z_2=\begin{bmatrix}
        0& 0\\
        \alpha&0
    \end{bmatrix}~\mbox{and}~
    Z_3=\begin{bmatrix}
        \alpha^2+1& 0\\
        0&1
    \end{bmatrix}.
\end{align}
It is clear that $Z$ is a positive and invertible operator. In the similar fashion as in Lemma \ref{subspaces with logmodularity properties cannot be in generic position}, it is easy to show, by using $p,q\in\M$, that $Z\in\M$.
Since $\A$ is logmodular in $\M$, we get a sequence $\{S_n\}$ of  invertible operators in $\A^{-1}$ such that $Z=\lim_nS_n^*S_n$. Then for each $n$, we have $S_np=pS_np$ and $S_n^{-1}p=pS_n^{-1}p$; hence the operators $S_n$ and $S^{-1}_n$ have the form
\begin{align}
    S_n=\begin{bmatrix}
        a_n&b_n&p_n&q_n\\
        c_n&d_n&r_n&s_n\\
        0&0&e_n&f_n\\
        0&0&g_n&h_n
    \end{bmatrix}
    = :\begin{bmatrix}
        A_n&B_n\\
        0&C_n
    \end{bmatrix}~~\mbox{and}~~
    S^{-1}_n=\begin{bmatrix}
        a'_n&b_n'&p_n'&q_n'\\
        c'_n&d'_n&r'_n&s_n'\\
        0&0&e_n'&f_n'\\
        0&0&g'_n&h'_n
    \end{bmatrix}
    =: \begin{bmatrix}
        A_n'&B_n'\\
        0&C_n'
    \end{bmatrix},
\end{align}
for appropriate operators $a_n,b_n,., a_n',b_n',..$ etc.
In particular, we have $A_nA'_n=1=A_n'A_n$ i.e. $A_n$ is invertible in $\B(\widetilde{\K}_1)$. Similarly $C_n$ is invertible in $\B(\widetilde{\K}_2)$.
Now
\begin{align}\label{Z=lim_nS_n*S_n}
   \begin{bmatrix}
       1&Z_2\\
       Z_2^*&Z_3
   \end{bmatrix}
   =Z=\lim_nS_n^*S_n  
    =\lim_n \begin{bmatrix}
       A^*_nA_n&A_n^*B_n\\
        B_n^*A_n&B_n^*B_n+C_n^*C_n
    \end{bmatrix}.
\end{align}
Then we have  $\lim_nA_n^*A_n=1$ and since $A_n$ is invertible, it follows  from Lemma \ref{in limits, invertible isometries behave like unitaries} that 
\begin{align}\label{lim_nA_nA_n*=1}
\lim_nA_nA_n^*=1.
\end{align}
We also have $\lim_nA_n^*B_n=Z_2$,  which after multiplied by $A_n$ to  left side of the sequence and using \eqref{lim_nA_nA_n*=1} yields  $\lim_n( B_n-A_nZ_2)=0$; but we have
\begin{align*}
    B_n-A_nZ_2=\begin{bmatrix}
        p_n&q_n\\
        r_n&s_n
    \end{bmatrix}-\begin{bmatrix}
        a_n&b_n\\
        c_n&d_n
    \end{bmatrix}\begin{bmatrix}
        0&0\\
        \alpha&0
    \end{bmatrix}
    =\begin{bmatrix}
        p_n-\alpha b_n&q_n\\
        r_n-\alpha d_n&s_n
    \end{bmatrix},
\end{align*}
and thus we get  the following equations:
\begin{align}
    &\lim_n(p_n-\alpha b_n)=0,\label{lim_n p_n-alpha b_n=0}\\
    &\lim_n(r_n-\alpha d_n)=0\label{lim_nr_n-alpha d_n=0}. 
\end{align}
Also if $D_n=B_n-A_nZ_2$ for all $n$, then $\lim_nD_n=0$ and since $\lim_nA_n^*A_n=1$, we have
\begin{align*}
    \lim_nB_n^*B_n=\lim_n(D_n+A_nZ_2)^*(D_n+A_nZ_2)=\lim_nZ_2^*A_n^*A_nZ_2=Z_2^*Z_2.
\end{align*}
Further, we have from \eqref{Z=lim_nS_n*S_n} that $\lim_n(B_n^*B_n+C_n^*C_n)=Z_3$ which yields
\begin{align}\label{lim_nC_n*C_n=1}
    \lim_nC_n^*C_n=Z_3-\lim_nB_n^*B_n=Z_3-Z_2^*Z_2=\begin{bmatrix}
      \alpha^2+1&0\\
      0&1
    \end{bmatrix}-\begin{bmatrix}
      \alpha^2&0\\
      0&0
    \end{bmatrix}=\begin{bmatrix}
      1&0\\
      0&1
    \end{bmatrix}.
\end{align}
In particular, by computing the matrices $C_n^*C_n$, we get $\lim_n(e_n^*e_n+g_n^*g_n)=1$;
hence there exists $m\in\N$ such that $\|e_n^*e_n\|\leq 2$, which in turn yields
\begin{align}\label{e_ne_n* leq 2}
    e_ne_n^*\leq 2~~\mbox{ for all }n\geq m.
\end{align}
Now
\begin{align*}
    S_nq=\begin{bmatrix}
        0&b_nx^2+p_nxy&b_nxy+p_ny^2&q_n\\
        0&d_nx^2+r_nxy&d_nxy+r_ny^2&s_n\\
        0&e_nxy&e_ny^2&f_n\\
        0&g_nxy&g_ny^2&h_n
    \end{bmatrix}
\end{align*}
and
\begin{align*}
    qS_nq=
    \begin{bmatrix}
        0&0&0&0\\
        0&x^2d_nx^2+x^2r_nxy+xye_nxy&x^2d_nxy+x^2r_ny^2+xye_ny^2&x^2s_n+xyf_n\\
        0&xyd_nx^2+xyr_nxy+y^2e_nxy&xyd_nxy+xyr_ny^2+y^2e_ny^2&xys_n+y^2f_n\\
        0&g_nxy&g_ny^2&h_n
    \end{bmatrix}.
\end{align*}
Since $S_nq=qS_nq$ for each $n$,  by equating $(3,2)$ entries of the respective matrices and then using $1-y^2=x^2$ we get the expression $x^2e_nxy=xyd_nx^2+xyr_nxy$;  but since  $x$ is one-one and hence $x$ has dense range, $x$ cancels from both sides of the equation to yield
\begin{align*}
    xe_ny=yd_nx+yr_ny.
\end{align*}
If we set $t_n=r_n-\alpha d_n$ for all $n$, then above equation further yields 
\begin{align*}
    xe_ny=yd_nx+yr_ny= yd_nx+y(\alpha d_n+t_n)y
=yd_n(x+\alpha y)+yt_ny,
\end{align*}
which in other words says that
\begin{align}\label{yd_n=xe_nyu^-1-..}
    yd_n=xe_nyu^{-1}-yt_nyu^{-1}
\end{align}
where $u=x+\alpha  y$, which  is clearly positive and invertible as $u^2\geq 1$. In a similar vein  as in \eqref{eq:y^2u^{-2}} in Lemma \ref{subspaces with logmodularity properties cannot be in generic position}, $u$ and $y$ commute and we have
\begin{align}\label{y^2u^-2 leq 1/alpha^2}
   y^2u^{-2}\leq 1/{\alpha^2}.
\end{align}
Also by  equating $(1,2)$ entries of  $S_nq$ and $qS_nq$, we get   $b_nx^2+p_nxy=0$;  again since $x$ has dense range, it follows that $ b_nx+p_ny=0$ for all $n$, which  by using \eqref{lim_n p_n-alpha b_n=0} further yields
\begin{align*}
    0=\lim_n(b_nx+p_ny)=\lim_nb_n(x+\alpha y)+\lim_n(p_n-\alpha b_n)y=\lim_nb_n(x+\alpha y).
\end{align*}
But $x+\alpha y$ is invertible as seen before,  so the above equation yields
 \begin{align}\label{lim_nb_n=0}
       \lim_nb_n=0.
 \end{align}
Similarly since $S_n^{-1}$ also has all these properties, we have
\begin{align}\label{lim_nb_n'=0}
    \lim_nb_n'=0.
\end{align}
Note that the $(2,2)$ entry of the matrix $S_nS_n^{-1}$ (with respect to the decomposition $\ran(p\wedge q^\perp)\oplus\K\oplus\K\oplus\ran(p^\perp\wedge q))$) is
$c_nb_n'+d_nd_n'$;  hence we have $c_nb_n'+d_nd_n'=1$ for all $n$. Since $\lim_nb_n'=0$ from \eqref{lim_nb_n'=0}, it follows that $ \lim_nd_nd_n'=1.$ Hence there exists $n_0\in\N$ such that $\|d_nd_n'-1\|<1$ for all $n\geq n_0$, which in particular says that $d_nd_n'$ is invertible for all $n\geq n_0$. Thus $d_nd_n'(d_nd_n')^{-1}=1$, which is to say that $d_n$ is right invertible for all $n\geq n_0$.
Likewise,  from $(2,2)$ entry of $S_n^{-1}S_n$ and using $\lim_nb_n=0$ from \eqref{lim_nb_n=0}, we get $\lim_nd_n'd_n=1$. Again this  implies that $d_n'd_n$ is invertible, and hence $d_n$ is left invertible for large $n$. Thus we have shown that $d_n$ is both left and right invertible, which is to say that $d_n$ is invertible for large $n$.

Now for each $n$, note that the $(2,2)$ entry of the matrix  $S_n^*S_n$ (with respect to the decomposition $\ran(p\wedge q^\perp)\oplus\K\oplus\K\oplus\ran(p^\perp\wedge q))$) is $b_n^*b_n+d_n^*d_n$. Since $\lim_nS_n^*S_n=Z$, it then follows that $\lim_n(b_n^*b_n+d_n^*d_n)=1,$ and since $\lim_nb_n=0$ from \eqref{lim_nb_n=0}, we get $\lim_nd_n^*d_n=1$. But $d_n$ is invertible for large $n$,  so it follows from Lemma \ref{in limits, invertible isometries behave like unitaries} that
\begin{align}\label{lim_nd_nd_n^*=1}
    \lim_nd_nd_n^*=1.
\end{align}
Now using $\lim_n t_n=0$ from \eqref{lim_nr_n-alpha d_n=0}, and equations \eqref{e_ne_n* leq 2}, \eqref{yd_n=xe_nyu^-1-..}, \eqref{y^2u^-2 leq 1/alpha^2} and \eqref{lim_nd_nd_n^*=1},  we get  the following:
\begin{align*}
    y^2&=\lim_nyd_nd^*_ny=\lim_n(yd_n)(yd_n)^*=\lim_n(xe_nyu^{-1}-yt_nyu^{-1})(xe_nyu^{-1}-yt_nyu^{-1})^*\\
    &=\lim_n(xe_nyu^{-1})(xe_nyu^{-1})^*=\lim_nxe_ny^2u^{-2}e_n^{*}x\leq\frac{1}{\alpha^2}\lim_nxe_ne^*_nx\leq \frac{2}{\alpha^2}x^2.
\end{align*}
Since $\alpha  \geq1 $ is arbitrary, it follows by taking $\alpha\to \infty$ that
$y=0$, which is  a contradiction. The proof is now complete.
\qed

\section{Reflexivity of algebras with factorization}

Our main result of this article says that lattice of any algebra with  factorization property in a factor is a nest.  A natural question that arises is whether  algebras with factorization  are also  nest subalgebras i.e. are they  reflexive? Certainly, we cannot always expect automatic reflexivity of such algebras (See Example \ref{an example of a subdiagonal algebra}).  But then what extra condition  can be imposed in order to show that they are reflexive?  A result due to Radjavi and Rosenthal \cite{RaRo} says that a weakly closed algebra in $\bh$ whose lattice is a nest, is a nest algebra if and only if it contains a maximal abelian self-adjoint algebra (masa). In this section, we show that if the lattice of  an algebra with factorization in $\bh$ has finite dimensional atoms, then it contains a masa and hence it is reflexive.  

We recall some terminologies to this end. Let $\M$ be a von Neumann algebra, and  let $\E$  be a complete  nest in $\M$.  For any projection $p\in\E$, let $$p_-=\vee\{q\in\E; q< p\}  \mbox{ and }p_+=\wedge\{q\in\E; q> p\}.$$
An {\em atom} of $\E$  is a nonzero projection of the form $p-p_-$ for some $p\in\E$  with $p\neq p_-$. Clearly  two distinct atoms are mutually orthogonal. The nest $\E$  is called {\em atomic}  if there is a finite or countably infinite sequence $\{r_n\}$ of atoms of $\E$  such that $\sum_nr_n=1$, where the sum converges in weak operator topology (WOT). Further an algebra $\A$ in $\M$ is called {\em $\M$-transitive} (simply {\em transitive} when $\M=\bh$) if $\Lat_\M\A=\{0,1\}$.  We now consider the following simple lemma.

\begin{lemma}\label{lattices of compressed algebras}
Let $\A$ be an algebra in a von Neumann algebra $\M$  such that $\Lat_\M\A$ is a nest, and let  $p,q\in \Lat_\M\A$ with $p<q$.  If $r=q-p$, then  $\Lat_{r\M r}(r\A r)=\{s\in r\M r; p+ s\in\Lat_\M\A \}$. In particular, if $p=q_-$ then $r\A r$ is $r\M r$-transitive.
\end{lemma}
\begin{proof}
As seen in Proposition \ref{logmodularity  is preserved on compression}, $r\A r$ is a subalgebra of $r\M r$. Now let $s\in\Lat_{r\M r}(r\A r)$, and let $a\in\A$. Note that $(rar)s=s(rar)s$, and  since $rs=s$, it follows that  $ras=(rar)s=s(rar)s=sas,$
using which, and the expression  $aq=qaq$, we have
\begin{align}\label{as=(p+s)as}
    as=aqs=qaqs=qas=pas+ras=pas+sas=(p+s)as.
\end{align}
Also since $sp=0$ and $ap=pap$, we have $sap=spap=0$, which along with \eqref{as=(p+s)as}  yield
\begin{align*}
    (p+s)a(p+s)=pap+sap+(p+s)as=ap+as=a(p+s).
\end{align*}
Since $a$ is arbitrary in $\A$, it follows that $p+s\in\Lat_\M\A$. Conversely let $s\in r\M r$ be a projection such that $p+s\in\Lat_\M\A$, and fix $a\in\A$. Then $a(p+s)=(p+s)a(p+s)$, and since $ps=0=pr$ and $rs=s$, we have
\begin{align*}
    (rar)s=ras=ra(p+s)s=r(p+s)a(p+s)s=s(rar)s.
\end{align*}
Again as $a\in\A$ is arbitrary, we conclude that $s\in \Lat_{r\M r}(r\A r)$. Thus we have proved the first assertion. Note that if $p=q_{-}$ then for any $s\in r\M r$,  $p+s\in\Lat_\M\A$  if and only if $s=0$ or $s=r$. The second assertion then follows from the first.
\end{proof}

The following lemma is the crux of this section. Recall our convention that all algebras are  unital and norm closed.

\begin{lemma}\label{atoms of finite range are in A}
Let an algebra $\A$ have factorization in a von Neumann algebra $\M$, and let $p,q\in\Lat_\M\A$ such that $p<q$. If $q-p$ has finite dimensional range, then $q-p\in\A$. In particular, if either $p$ or $p^\perp$ has finite dimensional range, then $p\in\A$.
\end{lemma}
\begin{proof}
The second assertion clearly follows from the first.  For the first assertion, set $r=q-p$. Let $\M$ be a von Neumann subalgebra of $\bh$ for some Hilbert space $\h$. Note that
$$\h=\ran(p)\oplus\ran(r)\oplus\ran(q^\perp),$$
and we consider operators of $\bh$ with respect to this decomposition.  For each $n\in\N$, set \begin{align*}
    X_n=r+\frac{1}{n}r^\perp=\begin{bmatrix}
    1/n&0&0\\
    0&1&0\\
    0&0&1/n
    \end{bmatrix}.
\end{align*}
It is clear that each $X_n$ is a positive and invertible operator, and since $r\in\M$ it follows that $X_n\in\M$. So by factorization property of $\A$ in $\M$, there exists an invertible operator $S_n\in\A^{-1}$ such that $X_n=S_n^*S_n$. Then each $S_n$ leaves $\ran(p)$ and $\ran(q)$ invariant, which equivalently says that $S_n$ has the form
\begin{align}
    S_n=\begin{bmatrix}
        a_n&b_n&c_n\\
        0&d_n&e_n\\
        0&0&f_n
    \end{bmatrix},
\end{align}
for appropriate operators $a_n, b_n..$ etc. We claim that  the off diagonal entries $b_n, c_n, e_n$ are $0$ for all $n$. Note that since  $S_n^{-1}\in\A$, each $S_n^{-1}$  leaves $\ran(p)$ and $\ran(q)$ invariant, meaning that $S_n^{-1}$ is also upper triangular. Consequently, the diagonal entries $a_n, d_n, f_n$ of $S_n$ are invertible. Now for all $n$, we have
\begin{align*}
  \begin{bmatrix}
    1/n&0&0\\
    0&1&0\\
    0&0&1/n
    \end{bmatrix}=X_n=S_n^*S_n=\begin{bmatrix}
    a_n^*a_n&a_n^*b_n&a_n^*c_n\\
    b_n^*a_n&b_n^*b_n+d_n^*d_n&b_n^*c_n+d_n^*e_n\\
    c_n^*a_n&c_n^*b_n+e_n^*d_n&c_n^*c_n+e_n^*e_n+f_n^*f_n
    \end{bmatrix}.
\end{align*}
By equating  $(1,2)$ entries of the matrices above, we have $a_n^*b_n=0$ and since $a_n$ is invertible, it follows that $b_n=0$. Similarly from $(1,3)$ entries, we have $a_n^*c_n=0$, from which we again use invertibility of $a_n$ to get  $c_n=0$. Further from $(2,3)$ entries,  we have $b_n^*c_n+d_n^*e_n=0$. But since $b_n=0$ and $d_n$ is invertible,  it follows that $e_n=0$. This proves the claim that for all $n$, the operators $b_n, c_n$ and $e_n$ are $0$.

Next equating $(1,1)$ entries of $S_n^*S_n$ and $X_n$ we have  $a_n^*a_n=1/n$ for all $n$, which implies that $\lim_na_n=0$. Also from equating $(3,3)$ entries we have $c_n^*c_n+e_n^*e_n+f_n^*f_n=1/n$, so it follows that $f_n^*f_n\leq 1/n$ for all $n$; hence  $\lim_nf_n=0$. Further since $\ran(r)$ is finite dimensional by hypothesis, from $d_n^*d_n=1$, $d_n$ is a unitary for every $n.$ By compactness of the unitary group in finite dimensions we get  a subsequence $\{d_{n_k}\}$ converging to a unitary $d\in \B(\ran(r)).$ Thus we have $\lim_kS_{n_k}=S$, where
$$S=\begin{bmatrix}
    0&0&0\\
    0&d&0\\
    0&0&0
\end{bmatrix}.$$
Since each $S_{n_k}\in\A$ and $\A$ is norm closed, it follows that $S\in\A$.  Note that  $\lim_kd_{n_k}^{-1}=\lim _kd_{n_k}^*= d^*=d^{-1}$, using which we have
\begin{align*}
    \lim_kS_{n_k}^{-1}S=\lim_k\begin{bmatrix}
       a_{n_k}^{-1}&0&0\\
        0&d_{n_k}^{-1}&0\\
        0&0&f_{n_k}^{-1}
    \end{bmatrix}\begin{bmatrix}
       0&0&0\\
        0&d&0\\
        0&0&0
    \end{bmatrix}=\lim_k\begin{bmatrix}
       0&0&0\\
       0&d_{n_k}^{*}d&0\\
        0&0&0
    \end{bmatrix}=\begin{bmatrix}
       0&0&0\\
        0&1&0\\
        0&0&0
    \end{bmatrix}=r.
\end{align*}
Since  $S_{n_k}^{-1}S\in\A$ (as $S_{n_k}^{-1}$ and $ S\in\A$) for all $k$, we conclude that $r\in\A$, as required to prove.
\end{proof}

We now discuss a sufficient criterion imposed on the dimension of atoms of the lattice to prove the reflexivity of an algebra having factorization in $\bh$. It is clearly not necessary as any nest algebra arising out of a countable nest has factorization and is reflexive.

To this end, let $\E$ be a complete nest in $\bh$. Let $\{r_n\}$ be the collection of all atoms of $\E$, and let $r=\sum_nr_n$ in WOT convergence. If $r\neq 1$, then it is straightforward to check that the nest $\{p\wedge r^\perp;\; p\in\E\}$ in $\mathcal{B}(\ran(r^\perp))$ is complete and  has no atom (such nests without any atom are called {\em continuous}). But then any continuous complete nest has to be uncountable (in fact indexed by $[0,1]$; see Lemma 13.3 in \cite{Da}). In particular, if the nest $\E$ is countable, then $r=1$ and hence $\E$ is atomic. Thus we have the following lemma:

\begin{lemma}\label{Lat A is atomic}
Let an algebra $\A$ have factorization  in $\bh$. Then $\Lat\A$ is an atomic nest.
\end{lemma}
\begin{proof}
Since $\A$ has factorization in $\bh$, $\Alg\Lat\A$  also has factorization in $\bh$ as it contains $\A$. Consequently $\Lat\A$  is a countable nest by Corollary \ref{strenghtening of Larson result},  so it is atomic.
\end{proof}

\begin{theorem}\label{algebras with finite atoms are reflexive}
Let $\A$ be a weakly closed algebra having factorization  in $\bh$. If all atoms of the lattice $\Lat\A$ has finite dimensional range, then $\A$ is reflexive and hence $\A$ is a nest algebra.
\end{theorem}
\begin{proof}
We shall show that $\A$ contains a masa. As noted above, this claim along with the fact that $\Lat\A$ is a nest (from Corollary \ref{lattices of algebras with factorization}) will imply  the required  assertion that $\A$ is reflexive and a nest algebra (see Theorem 9.24,  \cite{RaRo}). Let $\{r_i\}_{i\in\Lambda}$ be the collection of all the atoms of $\Lat\A$ for some at most countable indexing set $\Lambda$. Since $\Lat\A$ is atomic from Lemma \ref{Lat A is atomic}, it follows that $\sum_{i\in\Lambda}r_i=1$ in WOT;  hence $\h=\oplus_{i\in\Lambda}\h_i$, where $\h_i=\ran(r_i)$ which satisfies $\h_i\perp\h_j$ for all $i\neq j$. For each $i\in\Lambda$ since  $r_i$ is an atom, we  note that  $r_i=p_i-q_i$ for some $p_i, q_i\in\Lat\A$ (where $q_i={p_i}_-$), and since $r_i$ has finite dimensional range by hypothesis, it follows from Lemma \ref{atoms of finite range are in A} that $r_i\in\A$.

Now  recognize  the von Neumann algebra $r_i \bh r_i$ with $\B(\h_i)$ for each $i\in\Lambda$. Since $r_i$ is an atom, we know from  Lemma \ref{lattices of compressed algebras} that $r_i\A r_i$ is a transitive subalgebra of $\B(\h_i)$. Therefore, as $\h_i$ is finite-dimensional,  it follows from Burnside's Theorem (Corollary 8.6, \cite{RaRo}) that $r_i\A r_i=\B(\h_i)$. In particular, this implies that $r_i\bh r_i=r_i\A r_i$,  and since $r_i\in\A$, it follows that
\begin{align}\label{eq:r_i}
   r_i \bh r_i\subseteq\A.
\end{align}
Now for each $i$, let $\mathcal{L}_{i}$ be a masa in $\B(\h_i)$, and let $\mathcal{L}=\bigoplus_{i\in\Lambda}\mathcal{L}_{i}$, which is considered a subalgebra of $\bh$. It is clear that $\mathcal{L}$ is a masa in $\bh$. Note that $\mathcal{L}r_i=r_i\mathcal{L}$ for all $i\in\Lambda$.  Also it follows from \eqref{eq:r_i} that $\mathcal{L}_i=r_i\mathcal{L}r_i\subseteq\A$, and since $\A$ is WOT closed we have
$$\mathcal{L}=\sum_{i\in\Lambda}\mathcal{L}r_i=\sum_{i\in\Lambda}r_i\mathcal{L}r_i\subseteq\A,$$
where the sum above is taken in WOT. Thus we have shown our requirement that $\A$ contains a masa, completing the proof.
\end{proof}

A nest of projections on a Hilbert space is called {\em maximal or simple} if it is not contained in any larger nest. It is easy to verify that a nest $\E$ is maximal if and only if all atoms in  $\E$ are one-dimensional. Thus we have the following corollary.

\begin{corollary}\label{algebras whose lattice is maximal nest}
Let an algebra $\A$ have factorization in $\bh$, and let  $\Lat\A$ be a maximal nest. Then $\A$ is reflexive, and so it is a nest algebra.
\end{corollary}

We emphasize the importance of above corollary in the following example.

\begin{example}
Consider the Hilbert space $ \h=\ell^2(\Gamma) $, for $\Gamma=\N$ or $\mathbb{Z}$, and let $\mathcal{A}$ be the reflexive algebra of upper-triangular matrices in $\bh$ with respect to the  canonical basis $\{e_n\}_{n\in\Gamma}$. Note that $\Lat\A=\{p_n; n\in\Gamma\}$, where $p_n$ is the projection onto the subspace $\ospan\{e_m; m\leq n\}$. Clearly $\Lat\A$ is a maximal nest. So if $\mathcal{B}$ is any subalgebra of $\A$ with $\Lat\B$ a nest, then $\Lat\A\subseteq\Lat\B$, which implies by maximality  that $\Lat\A=\Lat\B$. Thus it follows from Corollary \ref{algebras whose lattice is maximal nest} that the only subalgebra of  $\A$ that has factorization in $\bh$ is $\A$.
\end{example}

Next we consider the consequence of above results for subalgebras of finite dimensional von Neumann algebras. Let $M_n$ denote the algebra of all $n\times n$ complex matrices for some natural number $n$.  Let $\A$ be a logmodular subalgebra of $M_n$. It can easily be verified by the compactness argument of unit ball of $M_n$ that the algebra $\A$ also has factorization in $M_n$.  Since all atoms of $\Lat\A$ are clearly finite dimensional, it follows from Theorem \ref{algebras with finite atoms are reflexive}  that $\A$ is a nest algebra in $M_n$. Thus we have shown that upto unitary equivalence, $\A$ is an algebra of block upper-triangular matrices in $M_n$. This assertion was put as a conjecture in \cite{PaRa}, and an affirmative answer was given in \cite{Ju}. We have provided a different solution, and  we state it  below.

\begin{corollary}
Let $\A$ be a logmodular algebra in $M_n$. Then $\A$ is an algebra of block upper triangular matrices upto unitary equivalence.
\end{corollary}

Moreover, the corollary above generalizes to any logmodular subalgebras of finite dimensional von Neumann algebras.

\begin{corollary}\label{logmodular subalgebra of finite dimensional algebra}
Let $\M$ be a finite dimensional von Neumann algebra, and let $\A$ be a logmodular algebra in $\M$. Then $\A$ is a nest subalgebra of $\M$, and $\A$ is $\M$-reflexive.
\end{corollary}
\begin{proof}
Since $\M$ is a finite dimensional von Neumann algebra, there exist  natural numbers $n_1,\ldots, n_k$ such that $\M$ is $*$-isomorphic to $M_{n_1}\oplus\ldots\oplus M_{n_k}$. In view of Proposition \ref{logmodularity are preserved under isomorphism}, we assume without loss of generality that $\M=M_{n_1}\oplus\ldots\oplus M_{n_k}$, which  acts on the Hilbert space $\h=\C^{n_1}\oplus\ldots\oplus\C^{n_k}$. Using compactness of the unit ball in  finite dimensional algebras, we note that $\A$ also has factorization  in $\M$.

Now for $i=1,\ldots,k$, let $p_i$ denote the orthogonal projection of $\h$ onto the subspace $\C^{n_i}$ (considered as a subspace of $\h$), and let $\A_i=p_i\A p_i$. We claim that $\A=\A_1\oplus\ldots\oplus\A_k$. Firstly note that $p_i\in\M\cap \M'$; hence $p_i\in\Lat_\M\A$. This in particular says that $\A_i$ is an algebra. Since $p_i$ has finite dimensional range, it follows from Lemma \ref{atoms of finite range are in A} that $p_i\in\A$. This implies that $\A_i\subseteq\A$ for each $i$; hence we have $\A_1\oplus\ldots\oplus\A_k\subseteq\A$. On the other hand, since   $\sum_{i=1}^kp_i=1$ we get
\begin{align*}
    \A=\A\sum_{i=1}^kp_i\subseteq\sum_{i=1}^k\A p_i=\sum_{i=1}^kp_i\A p_i=\oplus_{i=1}^k\A_i,
\end{align*}
proving our claim that $\A=\oplus_{i=1}^k\A_i$. Now if we recognize $M_{n_i}$ as a subalgebra of $\M$ ($*$-isomorphic to $p_i\M p_i$) for each $i$, then the algebra $\A_i$ has factorization in $M_{n_i}$ by Proposition \ref{logmodularity  is preserved on compression}; hence  the lattice  $\E_i=\Lat_{M_{n_i}}\A_i$ is a nest in $\C^{n_i}$ by Corollary \ref{lattices of algebras with factorization}, and that $\A_i=\Alg_{M_{n_i}}\E_i$ by Theorem \ref{algebras with finite atoms are reflexive}. Now consider the lattice
$$\E=\oplus_{i=1}^k\E_i=\{\oplus_{i=1}^kq_i;\; q_i\in\E_i,1\leq i\leq k\}$$
in $\M.$ It is clear that $\E=\Lat_\M\A$, which implies $\A\subseteq\Alg_\M\E$. Note that $\E$ is not a nest if $k\geq2$. Choose a sublattice, namely $\F$, of $\E$ such that $\F$ is a nest and each element $q_i$ in $\E_i$ appears at least once as the $i$th coordinate of an element of $\F$. Such $\F$ can always be chosen: for example consider the nest $\mathcal{F}_i$ for each $i$ given by
$$\mathcal{F}_i=\{e_1\oplus\ldots\oplus e_{i-1}\oplus q_{i}\oplus 0\oplus\ldots\oplus 0;\; q_i\in\E_i\}\subseteq\E,$$
where $e_i$ denotes the identity of $M_{n_i}$, and let $\mathcal{F}=\cup_i^k\mathcal{F}_i$. Since each $\E_i$ is a nest and $\mathcal{F}_1\subseteq\mathcal{F}_2\subseteq\ldots\subseteq\mathcal{F}_k$,  it follows that the sublattice $\mathcal{F}$ is a nest in $\M$, and note that $\mathcal{F}$ fulfils the requirement.

We claim that $\A=\Alg_\M\F$, which will prove that $\A$ is a nest subalgebra of $\M$. Clearly as $\mathcal{F}\subseteq\mathcal{E}$, we have $\A\subseteq \Alg_\M\E\subseteq\Alg_\M\mathcal{F}$. Conversely let $x\in\Alg_\M\F$, and let $x=\oplus_{i=1}^kx_i$ for some $x_i\in M_{n_i}$,  $1\leq i\leq k.$ The way $\mathcal{F}$ has been chosen, each element of $\E_i$ appears as the $i$th coordinate of some element of $\F$, so it follows that $ x_{i}q= qx_iq$ for all $q\in\E_i$ and for each $1\leq i\leq k$. This shows that $x_i\in \Alg_{M_{n_i}}\E_i=\A_i$; hence $x\in\A$. Thus we conclude that $\Alg_\M\mathcal{F}\subseteq\A$ proving the claim.

Finally since   $\F\subseteq\E=\Lat_\M\A$, it follows that $\Alg_\M\Lat_\M\A\subseteq\Alg_\M\F=\A$. Since the other inclusion is obvious, we have  $\A=\Alg_\M\Lat_\M\A$ which is to say that $\A$ is $\M$-reflexive.
\end{proof}

More generally, Corollary \ref{logmodular subalgebra of finite dimensional algebra} can easily be extended to any direct sum of finite-dimensional von Neumann algebras, whose proof goes on the same lines. We record the statement here.

\begin{corollary}\label{logmodular algberas of direct sum of finite dimensional algebras}
Let $\M$ be a (possibly countably infinite) direct sum of finite dimensional von Neumann algebras, and let $\A$ be a logmodular algebra in $\M$. Then $\A$ is a nest subalgebra of $\M$ and $\A$ is $\M$-reflexive.
\end{corollary}

In general, Corollary \ref{logmodular subalgebra of finite dimensional algebra} fails to be true for algebras having factorization (or logmodularity) in  infinite dimensional von Neumann algebra, as the following example suggests.

\begin{example}\label{an example of a subdiagonal algebra}
Let  $\A$ be an algebra having factorization in a von Neumann algebra $\M$  such that $\A\neq \M$ and $\mathcal{D}=\A\cap\A^*$ is a factor.  We claim that $\A$ is not $\M$-reflexive. Assume otherwise that $\A=\Alg_\M\Lat_\M\A$. Note that $\Lat_\M\A\subseteq\mathcal{D}$. Also it is easy to verify that $\Lat_\M\A\subseteq\mathcal{D}'$ and thus we have $\Lat_\M\A\subseteq \mathcal{D}\cap\mathcal{D}'=\C$. It then  follows that  $\Lat_\M\A=\{0,1\}$, so $\A=\Alg_\M\{0,1\}=\M$ which is  not true.

There are plenty of such algebras. To see one, let $G$ be a countable discrete ordered  group. Let $\ell^2(G)=\{f:G\to\mathbb{C}; \sum_{g\in G}|f(g)|^2<\infty\}$, and for each $g\in G$, let $U_g:\ell^2(G)\to \ell^2(G)$ be the unitary operator defined by $U_gf(g')=f(g^{-1}g')$ for $f\in \ell^2(G)$ and $g'\in G$. Let $\M$ be the finite von Neumann algebra in $\B(\ell^2(G))$ generated by the family $\{U_g\}_{g\in G}$, called the group von Neumann algebra of $G$. Note that each element $X$ of $\B(\ell^2(G))$ has a matrix representation $(x_{gh})$ with respect to the canonical basis of $\ell^2(G)$. Let
$$\A=\{X=(x_{gh})\in \M; x_{gh}=0\mbox{ for } g>h\}.$$
Then $\A$ is an example of a {\em finite maximal subdiagonal algebra} in $\M$ with respect to the expectation $\phi:\M\to\M$ given by $\phi((x_{gh}))=x_{ee}1$ for $(x_{gh})\in\M$, where $e$ denotes the identity of $G$ (see Example 3, \cite{Ar1}). In particular, $\A$ has factorization in $\M$ (Theorem 4.2.1, \cite{Ar1}). But note that $\A\cap\A^*=\C$ (in fact if $X=(x_{gh})\in\A\cap\A^*$, then $x_{gh}=0$ for all $g\neq h$ and $x_{gg}=x_{g'g'}$ for all $g,g'\in G$), so $\A$ cannot be $\M$-reflexive as discussed above. Moreover, we can  choose the ordered group $G$ to be  countable  with ICC property (e.g. $G=\mathbb{F}_2,$ the free group on two generators), so that $\M$ is a factor. In this case although $\Lat_\M\A$ is a nest, $\A$ cannot be a nest subalgebra of $\M$, otherwise  $\A\cap\A^* $ will contain the nest and so cannot be equal to $\C$.
\end{example}

\section{Concluding remarks}

In this paper we have discussed `universal' or `strong' factorization property for  subalgebras of  von Neumann algebras. But there are weaker notions of factorization which can also be explored. Say a subalgebra $\A$ has {\em weak factorization property (WFP)} in a von Neumann algebra $\M$ if for any positive element $x\in\M$, there is an element $a\in\A$ such that $x=a^*a$. Here the invertibility assumption is relaxed.

Power \cite{Po} has  studied WFP of nest algebras where he  proved that if a nest $\E$ of projections on a Hilbert space $\h$ is well-ordered (i.e. $p\neq p_+=\cap_{q>p}q$ for all $p\in\E$ and $p\neq 1$), then $\Alg\E$ has WFP in $\bh$. Inspired from our result on lattices of algebras with factorization, it appears that lattices of  algebras with WFP in a factor  should also be a nest. But it is not clear to us at this point. However, for a subalgebra in a finite von Neumann algebra we can certainly say so.
We can follow the similar lines of proof along with the fact  that any left (or right) invertible element in a finite von Neumann algebra is invertible. We record it here.

\begin{theorem}
Let $\A$ be a subalgebra  of a finite von Neumann algebra (resp. factor) $\M$ satisfying WFP. Then $\Lat_\M\A$ is a commutative subspace lattice (resp. nest). 
\end{theorem}

So a natural question is the following:

\begin{question}
Is the lattice of a subalgebra satisfying WFP in a von Neumann algebra (resp. factor) is a commutative subspace lattice (resp. nest)?
\end{question}

We conclude with the question of reflexivity of algebras with factorization. We showed that an algebra with factorization in $\bh$ has masa and hence is reflexive if we impose some dimensionality condition on the atoms of its lattice. But we still do not know whether every algebra with factorization in $\bh$ has masa. Thus the following question is open.

\begin{question}
Is a weakly closed algebra having factorization in $\bh$  automatically reflexive? In particular, is a weakly closed transitive algebra with factorization  equal to $\bh$?
\end{question}

\begin{acknowledgment*} The first author thanks SERB(India) for financial support through J C Bose Fellowship.

\end{acknowledgment*}

\end{document}